\documentclass{amsart}
\usepackage{graphicx} % Required for inserting images
\usepackage{tikz}
\usepackage[T1]{fontenc}     % <<--- important: enables ogonek glyphs
\usepackage[utf8]{inputenc} 
\newtheorem{theorem}{Theorem}

\newtheorem{lemma}[theorem]{Lemma}
\newtheorem{definition}[theorem]{Definition}
\newtheoremstyle{myremark} % name
  {}                        % space above
  {}                        % space below
  {\normalfont}             % body font
  {}                        % indent (empty = no indent)
  {\bfseries}               % theorem head font
  {.}                       % punctuation after theorem head
  { }                       % space after theorem head
  {}                        % theorem head spec

\theoremstyle{myremark}

\title{Finite stick-breaking sequences}
\author{Jan-Tino Brethouwer \and Anne Buijsrogge \and Robbert Fokkink }
\address{Department of Applied Mathematics, Technical University Delft, Netherlands.}
\email{r.j.fokkink@tudelft.nl}
\date{\today}
\begin{document}

\begin{abstract}
    In 1949, De Bruijn and Erd\H{o}s considered infinite stick-breaking
    sequences. 
    In this work, we consider finite sequences and prove that an infinite
    sequence that is due to De Bruijn, Erd\H{o}s, and Toulmin is in some
    sense unique.
\end{abstract}

\maketitle

\section{Introduction}

Take a stick of length one. Break it into two pieces. Repeat. At each step, take one of the pieces and break it. At step \(n\) there are \(n\) pieces, and your job is to keep the ratio between the shortest and the longest piece as close to \(1\) as possible. What is the best ratio between the shortest and longest pieces that you can maintain throughout this process?

This problem was introduced and solved by De Bruijn and
Erd\H{o}s~\cite{deBE}. It was revisited by Ramshaw~\cite{ramshaw}
and, more recently, by Anholcer et al.~\cite{anholcer}, who connected
it to a problem of Steinhaus on fair division~\cite{steinhaus1949}. De Bruijn and Erd\H{o}s also considered
a more general problem, taken up recently by Clément and Steinerberger
and by Korsky~\cite{clement2025balanced,korsky2026resolution}.
In a different direction, Kakutani proposed a random stick-breaking
process. His conjecture about the distribution of the break points was
proved by van Zwet; Pyke and van Zwet later obtained further convergence
results~\cite{pykevanzwet}. 
Carbone, Iaco, and Vol{\v{c}}i{\v{c}} considered the ergodic properties
of Kakutani's 
stick-breaking process~\cite{carbone2014}.

We consider a finite version of the stick-breaking problem.
Suppose the process ends after $N$ breaks. What is the best ratio
we can maintain at every step? 
That is what led us to consider finite stick-breaking sequences.
This question was inspired by a
social-distancing experiment. People arrive at a
round table one by one and choose where to sit without talking to each other.
How far apart can they manage to sit? This mimics a practical question that arose
during the COVID-19 pandemic: were assigned seats in cinemas and concert halls
really necessary~\cite{buijsrogge2026self}?

\section{Definitions and notation}

Formally, we consider a sequence $(a_n)$ of points on the unit circle, which we identify with the interval $[0,1)$ modulo 1, with the standard metric.
To simplify notation, we often suppress the index and write a sequence simply as $(a)$, reserving $n$ for the number of breaks.
Placing the first point $a_1$ makes a cut that opens the circle and turns it into an interval.
After the first $n$ points $a_1, a_2, \ldots, a_n$ are chosen, the circle is partitioned into $n$ half-open
intervals, or \emph{sticks}. We denote by $M_n(a)$ and $m_n(a)$ the lengths of the longest and shortest sticks, respectively. 
Our goal is to keep the ratio 
\[\rho_n(a)=\frac{M_n(a)}{m_n(a)}\] 
as close to $1$ as possible. 

The average stick length after $n$ breaks is $\frac 1n$, which we rescale to unit length.
For the rescaled maximum and minimum lengths
$n\cdot M_n(a)$ and $n\cdot m_n(a)$,
De Bruijn and Erd\H{o}s~\cite{deBE} proved that  
\begin{equation}\label{eq:1}
\limsup_{n\to\infty} {n}{\cdot} {M_n(a)}\geq \frac 1{\log 2}
\end{equation}
and
\begin{equation}\label{eq:2}
\liminf_{n\to\infty} {n}{\cdot}{m_n(a)}\leq \frac 1{2\log 2}.
\end{equation}
Therefore, it is impossible to break the sticks in such a way that the ratio
$
\rho_n(a)
$
converges to 1. De Bruijn and Erd\H{o}s proved that,
for any sequence $(a)$, the best possible asymptotic ratio is $2$
\begin{equation}\label{eq:3}
\limsup_{n\to\infty} \rho_n(a) \geq 2.
\end{equation}
There are many sequences $(a_n)$ such that $\rho_n(a)\leq 2$ for all $n$.
Indeed, any finite sequence $a_1,\ldots,a_n$ that satisfies this bound
can be extended indefinitely by always breaking a longest stick into two
equal halves. Maintaining the strict inequality $\rho_n(a)<2$ for all $n$
is harder. Harder still is to satisfy the three bounds all at once:
\begin{equation}\label{ineq:all3}
\boxed{
\rho_n(a)<2,\qquad
nM_n(a)<\frac1{\log 2},\qquad
nm_n(a)>\frac1{2\log 2}.}
\end{equation}
De Bruijn and Erd\H{o}s proved that the sequence
\begin{equation}\label{eq:BERT}
b_n=\log_2(2n\!-\!1)\pmod 1
\end{equation}
has this property. Toulmin rediscovered the result some years
later~\cite{toulmin}, so we call this the De Bruijn--Erd\H{o}s--Toulmin
(BET) sequence.

A few decades later, Niederreiter~\cite{niederreiter1984} rediscovered the inequalities of De Bruijn and Erd\H{o}s and claimed them as his own~\cite[p.~153]{niederreiter1992}. 
He attributed the BET sequence to Ruzsa. We could have called it
the BERT sequence. 
We will prove that BET is the only greedy stick-breaking sequence
satisfying all three strict inequalities. We do not know whether the
greedy assumption is needed for uniqueness. Ramshaw showed that if
the inequalities are made non-strict, infinitely many sequences
satisfy them.

Steinhaus~\cite[Problem 7]{steinhaus} asked if it is possible to find $N$ numbers $x_1,x_2,\ldots,x_N$ 
in the unit interval $[0,1]$ so
that the first two lie in different halves, the first three in
different thirds, etc.
Warmus~\cite{warmus} proved that this is only possible up to $N=17$.
Anholcer et al.~\cite{anholcer} modified the problem and 
considered the minimum number $f(n)$ of breaks such that
all sticks have length $\leq \tfrac 1n$. They proved that
the BET sequence is optimal.

\section{A review of the results of de Bruijn and Erd\H{o}s}

The following lemmas, due to de Bruijn and Erd\H{o}s~\cite{deBE},
give elegant proofs of inequalities~(\ref{eq:1}), (\ref{eq:2}) and~(\ref{eq:3}).

\begin{lemma}[De Bruijn and Erd\H{o}s]\label{lem:1}
    Let $H_n=\frac 11 +\frac 12+\cdots +\frac 1n$ be the $n$-th harmonic
    number and let $H^m_n=H_m-H_n$.
    Let \[\Lambda_n(a)=\sup\{i\cdot  M_i(a), n/2< i\leq n\}\] and
    \[\lambda_n(a)=\inf\{i\cdot m_i(a), n/2< i\leq n\}.\]
    Then $\Lambda_{2n\!-\!1}(a)\geq 1/H^{2n\!-\!1}_{n-1}$ and $\lambda_{2n}(a)\leq 1/2H^{2n}_{n}$. 
    The inequalities~(\ref{eq:1}) and~(\ref{eq:2}) follow from standard properties of
    the harmonic numbers. The sequence $H^{2n\!-\!1}_{n-1}$ is descending and $H^{2n}_{n}$ is ascending, with limit ${\log 2}$.
\end{lemma}
\begin{proof}
    Let $\ell_1\leq \ell_2\leq \cdots\leq\ell_n$ be the lengths of the sticks after
    breaking $n$ times, sorted in ascending order. They add up to one, and $M_n(a)$ is equal
    to $\ell_n$. 
    The maximal stick length only decreases if we break the longest stick.
    It follows that 
    \begin{equation}\label{ineq:M_n}
        M_{n+k}(a)\geq \ell_{n-k},
    \end{equation} 
    for $k=0,\ldots,n-1$.
    Since the $\ell_i$ add up
    to one, we have 
    \begin{equation}\label{cut}
    M_n(a)+M_{n+1}(a)+\cdots+M_{2n\!-\!1}(a)\geq 1.
    \end{equation}
    By definition 
    \begin{equation}\label{ineq:Lambda}\Lambda_{2n\!-\!1}(a)\geq (n+k)\cdot M_{n+k}(a)\end{equation}
    for all $k=0,\ldots,n-1$ which implies
    that
    \begin{equation}\label{Lambdabound}
    \Lambda_{2n\!-\!1}(a)\cdot H^{2n\!-\!1}_{n-1}=\sum_{k={0}}^{n-1}\frac {\Lambda_{2n\!-\!1}(a)}{n+k}\geq 1.
    \end{equation}
    This gives the required lower bound on $\Lambda_{2n\!-\!1}(a)$.
    
    To bound $\lambda_{2n}(a)$ we work in the opposite direction, gluing broken sticks
    back together again.
    Let $\ell_1\leq \ell_2\leq \cdots\leq\ell_{2n}$ be the lengths of the sticks after
    breaking $2n$ times, sorted in ascending order. In particular, $m_{2n}(a)$
    is equal to $\ell_1$. We now work our way back, each step gluing two
    sticks. We have that 
    \begin{equation}\label{ineq:m_n}
    m_{2n-k}\leq \ell_{2k+1}        
    \end{equation}
    for $k=0,\ldots,n\!-\!1$
    because at least one of the $2k+1$
    shortest sticks remains unglued. We have
    \begin{equation}\label{glue}
    m_{2n}(a)+m_{2n\!-\!1}(a)+\cdots+m_{n+1}(a)\leq \ell_1+\ell_3+\cdots+\ell_{2n\!-\!1}\leq \frac 12.      
    \end{equation}
    By definition $\lambda_{2n}(a)/k\leq m_k(a)$ for all $n<k\leq 2n$ which implies
    that
    \begin{equation}\label{lambdabound2}
    \lambda_{2n}(a)\cdot H^{2n}_{n}\leq \frac 12.
    \end{equation}
        This gives the required upper bound on $\lambda_{2n}(a)$.
\end{proof}

Inequality~(\ref{eq:3}) follows from the following lemma.

\begin{lemma}[De Bruijn and Erd\H{o}s]\label{lem:2}
    For all $c<2$ and all $N$ there exists an $n\geq N$ such that $\rho_n(a)>c$.
\end{lemma}
\begin{proof}
    When a stick is broken, the shorter piece has length at most half of the original. It follows that $2m_{n+1}(a)\leq M_{n}(a)$. If we assume that 
    $\frac {M_n(a)}{m_n(a)}\leq c$ for all $n\geq N$, then using $M_n\geq 1/n$ we get 
    \begin{equation}\label{telescopic}        
    \frac{2^{n-N}}{n}\leq \frac {M_{n}2^{n-N}}{m_{N}}\leq \frac {M_N}{m_N}\cdot \frac {M_{N+1}}{m_{N+1}}\cdots\frac {M_{n}}{m_{n}}\leq c^{n+1-N}.
    \end{equation}
    Clearly, this is only possible if $c\geq 2$.
\end{proof}

Recall that the BET sequence is given by
$b_j=\log_2(2j-1)\pmod{1}$.
Multiplying $2j-1$ by a power of $2$ adds an integer to its
logarithm, so $b_j$ remains unchanged modulo~$1$.
For each odd number $1,3,\ldots,2n\!-\!1$, there is a unique
nonnegative integer $k$ such that
\[
n\leq 2^k(2j-1)<2n.
\]
These representatives are distinct, since every positive integer
has a unique odd part. Thus they are precisely the integers
$n,n+1,\ldots,2n\!-\!1$, and the breakpoints $b_1,\ldots,b_n$ are
\[
\log_2(n),\; \log_2(n+1),\; \ldots,\;
\log_2(2n\!-\!1)\pmod{1},
\]
possibly in a different order.

The map $\log_2(x)$ sends $[n,2n)$ bijectively onto
the unit circle $\pmod{1}$, preserving cyclic order. Thus these points are
the endpoints of the first $n$ sticks. Starting at $\log_2(n)$,
the first stick, $[\log_2(n),\log_2(n+1))$, is the longest,
and the last, $[\log_2(2n\!-\!1),\log_2(2n))$, is the shortest.

The next breakpoint is
\[
b_{n+1}=\log_2(2n+1)
       =\log_2\!\left(n+\tfrac12\right)\pmod{1}.
\]
Since $n<n+\tfrac12<n+1$, it splits the longest stick.
The two resulting sticks are shorter than all the remaining sticks.
We say that such a procedure is greedy.

\begin{definition}
A stick-breaking procedure is \emph{greedy} if, at every step,
it breaks a longest stick into two pieces. Each new piece must 
be at most as long as every
remaining stick, but at least as long as half the length of the
longest remaining stick.
\end{definition}

Greedy stick-breaking maintains the bound $\rho_n\leq2$.
If the first break divides the stick into two equal pieces,
then all subsequent stick lengths are determined:
\[
\left\{\frac12,\frac12\right\},\quad
\left\{\frac12,\frac14,\frac14\right\},\quad
\left\{\frac14,\frac14,\frac14,\frac14\right\},\quad\ldots.
\]
The ratio first reaches $2$ when there are three sticks, so
$\rho_n=2$ for all $n\geq3$.
These lengths are the same as those obtained from the
classical
van der Corput sequence~\cite{vanderCorput1935},
which remains a subject of study today~\cite{spiegelhofer}.

For the BET sequence we get that
\[
M_n(b)= \log_{2}\!\Big(1+\tfrac{1}{n}\Big),
\qquad
m_n(b)= -\log_{2}\!\Big( 1-\tfrac{1}{2n}\Big).
\]
For the BET sequence $n\cdot M_n(b)$ is 
increasing and $n\cdot m_n(b)$ is decreasing.
It immediately follows that the three 
inequalities in (\ref{ineq:all3}) are all satisfied
during the entire procedure.

\section{Finite BET sequences}

Since we are interested in stick lengths rather than break points,
we consider sequences of breakpoints to be the same if they produce the 
same stick lengths. 
\medbreak

We write $\Lambda_n(a)$ and $\lambda_n(a)$ for the scaled maximum and minimum
lengths up to $n$ breaks, which is slightly different from our notation
in Lemma~\ref{lem:1}. In particular,
\begin{align*}
\Lambda_n(a) &= \max\{\, i \cdot M_i(a) : i \leq n \,\}, \\
\lambda_n(a) &= \min\{\, i \cdot m_i(a) : i \leq n \,\}, \\
\rho_n(a)    &= \max\{\, M_i(a)/m_i(a) : i \leq n \,\}.
\end{align*}
According to Lemma~\ref{lem:1} we have
\[
\Lambda_{2n\!-\!1}(a)\geq \frac 1{H_{n-1}^{2n\!-\!1}},\;
\lambda_{2n}(a)\leq \frac 1{2H_{n}^{2n}}.
\]
How sharp are these inequalities?

Let us try to find a sequence $a_1,a_2,a_3$ such that
$\Lambda_3(a)=1/H^3_1=6/5$. There are infinitely many such
sequences. For example, we may take
\[
a_1=0,\qquad a_2=\frac35,\qquad
a_3=x\in\left[\frac15,\frac25\right].
\]
Then $2M_2(a)=3M_3(a)=6/5$, as required.
\newline
What about a sequence $a_1,a_2,a_3,a_4$ such that $\lambda_4(a)=1/2H^4_2=\frac {6}7$? 
It turns out that
there is only one way to do this
{\small
\[
0,\; \frac 47,\; \frac 27,\; \frac {11}{14}.
\]}
\vspace{-0.2cm}
\newline
For this sequence $a_1,a_2,a_3,a_4$ we have a suboptimal $\Lambda_3(a)=\frac 97 > 1/H^3_1$ and $\lambda_2(a)=\frac 67 < 1/(2H^2_1)$. 
It immediately follows at step four that there is no infinite sequence $(a)$ such that, for all $n$,  
\[
\lambda_{2n}(a) = \frac 1{2H_{n}^{\,2n}}.
\]  
The best one can hope for is the optimality at a fixed number of breakpoints.

\begin{definition}
A sequence \( a_1, \ldots, a_{2n\!-\!1} \) of odd length is called a \emph{\( (2n\!-\!1) \)-BET sequence} if it satisfies
\[
\Lambda_{2n\!-\!1}(a) = \frac{1}{H_{n-1}^{2n\!-\!1}}\;\text{ and }\;\rho_{k}(a)<2\text{ for all }k\leq 2n\!-\!1.
\]
Similarly, a sequence \( a_1, \ldots, a_{2n} \) of even length is called a \emph{\( 2n \)-BET sequence} if it satisfies
\[
\lambda_{2n}(a) = \frac{1}{2 H_{n}^{2n}}\;\text{ and }\;\rho_{k}(a)<2\text{ for all }k\leq 2n.
\]
\end{definition}
To abbreviate the notation, we write
\begin{equation}\label{eq:Lambdalambda}
\Lambda_{2n\!-\!1}=\frac 1{H_{n-1}^{2n\!-\!1}},\ \lambda_{2n}=\frac{1}{2H_n^{2n}}.
\end{equation}
Suppose that we want to construct a $(2n\!-\!1)$-BET sequence.
We need to control the maximum stick length, and keep the
ratio between longest and shortest stick below 2.
Assume that after $n$ breaks the stick lengths are
\begin{equation}\label{nsticks}
\frac{\Lambda_{2n\!-\!1}}{n},\;
\frac{\Lambda_{2n\!-\!1}}{n+1},\;
\ldots,\;
\frac{\Lambda_{2n\!-\!1}}{2n\!-\!1}.
\end{equation}
We still need $n-1$ breaks to complete the construction.
At each step, greedily break the longest remaining stick into two parts
that are smaller than the smallest stick.
After $n-1$ further breaks, we obtain a $(2n\!-\!1)$-BET sequence.
These further breaks can be chosen in many ways. However, the
stick lengths at the intermediate stage are uniquely determined.
Here, uniqueness is understood up to the equivalence introduced at the beginning of this section: sequences are identified if they produce the same stick lengths at every step.

\begin{theorem}\label{thm:fin1}
Every $(2n\!-\!1)$-BET sequence has the same stick lengths
after $n$ breaks, namely the lengths in~\eqref{nsticks}.
\end{theorem}

\begin{proof}
Suppose that $a_1,\ldots,a_{2n\!-\!1}$ is a $(2n\!-\!1)$-BET sequence.
Then inequalities~\eqref{ineq:M_n} and~\eqref{ineq:Lambda}
in the proof of Lemma~\ref{lem:1} must both be equalities.
Using the notation of that proof, Writing $\Lambda=\Lambda_{2n\!-\!1}$, we obtain
\[
\ell_{n-k}=M_{n+k}(a)
=\frac{\Lambda}{n+k},
\qquad k=0,\ldots,n-1.
\]
Thus, after the first $n$ breaks, the stick lengths are precisely
those in~\eqref{nsticks}.

To complete the construction, we show that these lengths can
indeed be attained while respecting
$kM_k(a)\leq\Lambda$ at every earlier step.
Starting with the $n$ sticks in~\eqref{nsticks}, we work backwards,
repeatedly gluing the two shortest sticks until only the original
stick remains.
The ratio between the longest and shortest sticks remains
strictly less than $2$. Indeed, the glued stick is longer than
every remaining stick, but shorter than twice the shortest
remaining stick.

 Suppose that the current $k$
stick lengths, in decreasing order, are bounded term by term by
\[
\frac{\Lambda}{k},\;
\frac{\Lambda}{k+1},\;
\ldots,\;
\frac{\Lambda}{2k-1}.
\]
The sum of the two shortest lengths is at most
\[
\frac{\Lambda}{2k-2}+\frac{\Lambda}{2k-1}
<\frac{\Lambda}{k-1}.
\]
Since the glued stick becomes the longest, the resulting
$k-1$ lengths are bounded term by term by
\[
\frac{\Lambda}{k-1},\;
\frac{\Lambda}{k},\;
\ldots,\;
\frac{\Lambda}{2k-3}.
\]
Induction therefore gives $kM_k(a)\leq\Lambda$ for
$k=1,\ldots,n$. Reversing the gluing procedure gives the
required sequence of breaks.
\end{proof}

The $2n$-BET sequence can be constructed by greedy stick-gluing from sticks of lengths 
\begin{equation}\label{2nsticks}
\frac{\lambda_{2n}}{n+1}\;,\;\frac{\lambda_{2n}}{n+1}\;,\;\frac{\lambda_{2n}}{n+2}\;,\;\frac{\lambda_{2n}}{n+2}\;,\;\ldots\;,\;\frac{\lambda_{2n}}{2n}\;,\;\frac{\lambda_{2n}}{2n},
\end{equation}
with $\lambda_{2n}$ as in Equation (\ref{eq:Lambdalambda}).

\begin{theorem}\label{thm:fin2}
    There is only one way to produce a $2n$-BET sequence and it is greedy. 
\end{theorem}
\begin{proof}
    Suppose that $a_1,\ldots,a_{2n}$ is a $2n$-BET sequence.
    This implies that inequality~(\ref{glue}) in the proof of Lemma~\ref{lem:1}
    is an equality:
    \[\ell_1+\ell_3+\cdots+\ell_{2n\!-\!1}=\frac 12.\]
    This implies that 
    \[
    \ell_2+\ell_4+\cdots+\ell_{2n}=\frac 12
    \]
    as well.
    In particular, $\ell_{2j-1}=\ell_{2j}$ and $m_{2n-k}=\ell_{2k+1}$
    for $k=0,\ldots,n-1$. The inequality in (\ref{lambdabound2}) must
    be an equality and \[j\cdot m_{j}=\frac{1}{2H_n^{2n}}\] for $j=n+1,\ldots,2n$.
    It follows that the $2n$-tuple of stick-lengths after
    $2n$ breaks must be as in (\ref{2nsticks}). If we work
    back from $2n$ to $n$ breaks by gluing sticks, this can only be done 
    by gluing the two smallest sticks to maintain the ratio~$<2$.

   We show that the only way to work back to the first break
is by greedy stick-gluing, proving uniqueness of the
$2n$-BET sequence.
Write $\lambda=\lambda_{2n}$ and let
$(\ell_1^m,\ldots,\ell_m^m)$ be the stick lengths after
$m$ breaks, in decreasing order.
Each gluing step must preserve the bound
$m\ell_m^m\geq\lambda$.
We have already proved that
\[
(\ell_1^n,\ldots,\ell_n^n)
=
\left(
\frac{2\lambda}{n+1},
\frac{2\lambda}{n+2},
\ldots,
\frac{2\lambda}{2n}
\right).
\]
    We argue by induction backwards from $n$ to $1$ that
\begin{equation}\label{eq:bounds}
\frac{2\lambda}{m+j-1}>\ell_j^m
\geq\frac{2\lambda}{m+j},
\qquad j=1,\ldots,m.
\end{equation}
These bounds hold at $m=n$. Suppose that they hold for some
$m\geq2$. To return to $m-1$ breaks, we must glue two sticks.
Both of the two shortest sticks have length strictly less than
$\lambda/(m-1)$, since
\[
\ell_m^m\leq\ell_{m-1}^m
<\frac{2\lambda}{2m-2}
=\frac{\lambda}{m-1}.
\]
Leaving either one unglued would therefore violate
$(m-1)\ell_{m-1}^{m-1}\geq\lambda$.
Thus, we must glue the two shortest sticks.
Their combined length satisfies
\[
\frac{2\lambda}{m-1}
>
\ell_{m-1}^m+\ell_m^m
>
\frac{2\lambda}{m}
>
\ell_1^m.
\]
Here the first two inequalities follow from
\[
\frac{2\lambda}{2m-2}+\frac{2\lambda}{2m-1}
>
\ell_{m-1}^m+\ell_m^m
\geq
\frac{2\lambda}{2m-1}+\frac{2\lambda}{2m}.
\]
Hence, the glued stick becomes the longest, and
\[
\ell_1^{m-1}=\ell_{m-1}^m+\ell_m^m,
\qquad
\ell_j^{m-1}=\ell_{j-1}^m
\quad (j=2,\ldots,m-1).
\]
The bounds~\eqref{eq:bounds} follow directly for $j>1$,
and the displayed inequalities establish them for $j=1$.
This completes the induction.
\end{proof}

\section{Uniqueness of the infinite BET sequence}
We now consider infinite sequences $(a)$ that satisfy all three bounds~(\ref{ineq:all3}).
Let $I_1$ denote the unit interval, and let $I_2, I_3, \ldots$ be the subsequent sticks produced by this process. 
These can be arranged in strictly decreasing order of length, so that 
\(\ell(I_1) > \ell(I_2) > \ell(I_3) > \cdots\). 
We denote their lengths by \(\ell_1, \ell_2, \ldots\).

\begin{lemma}
Suppose that a sequence $(a)$ satisfies
\begin{equation}\label{eq:all3}
\rho_n(a)<2,\qquad
nM_n(a)<\frac{1}{\log 2},\qquad
nm_n(a)>\frac{1}{2\log 2}
\end{equation}
for all $n$.
Let $\ell_n$ be the length of the longest stick after $n$ breaks.
Then
\[
\frac{1}{n\log 2}>\ell_n>
\frac{1}{(n+1)\log 2}.
\]
\end{lemma}

\begin{proof}
At each step, the longest stick must be broken.
Indeed, if a stick of length $L$ is broken while a stick of
length $M\geq L$ remains, one of the new pieces has length
at most $L/2$. The resulting ratio is therefore at least
$2M/L\geq2$, contrary to~\eqref{eq:all3}.

Since $\ell_n=M_n(a)$, the upper bound follows directly
from~\eqref{eq:all3}.
The next break splits this longest stick into two pieces,
each of length at least $m_{n+1}(a)$. Hence
\[
\ell_n\geq2m_{n+1}(a)>
\frac{1}{(n+1)\log 2},
\]
which proves the lower bound.
\end{proof}

Stick-breaking can be represented by an infinite binary tree.
Each node represents a stick, and its two children represent
the pieces produced when it is broken.
For greedy stick-breaking, the nodes are indexed in
breadth-first order. In particular, the children of
stick $n$ are sticks $2n$ and $2n+1$.
We call a stick-breaking sequence \emph{eventually greedy}
if this description holds for all sufficiently large $n$.
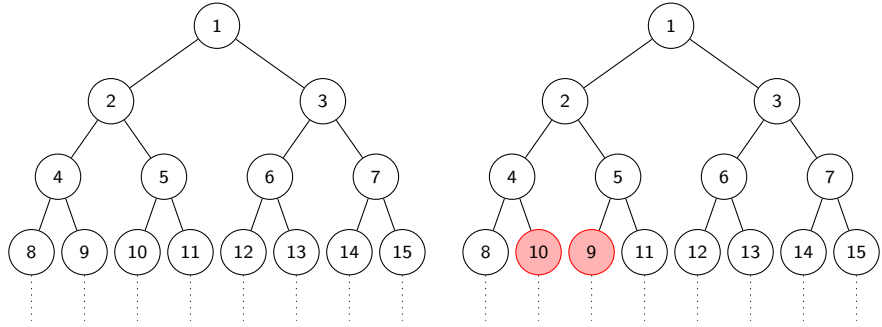
\begin{figure}[ht]
\centering
\begin{tikzpicture}[
  level distance=10mm,
  level 1/.style={sibling distance=28mm},
  level 2/.style={sibling distance=14mm},
  level 3/.style={sibling distance=7mm},
  every node/.style={circle,draw,minimum size=6mm,inner sep=0,font=\sffamily\scriptsize}
]

% -------------------------
% First tree (left)
% -------------------------
\node (r1) {1}
  child { node (n2) {2}
    child { node (n4) {4}
      child { node (a8)  {8} }
      child { node (a9)  {9} }
    }
    child { node (n5) {5}
      child { node (a10) {10} }
      child { node (a11) {11} }
    }
  }
  child { node (n3) {3}
    child { node (n6) {6}
      child { node (a12) {12} }
      child { node (a13) {13} }
    }
    child { node (n7) {7}
      child { node (a14) {14} }
      child { node (a15) {15} }
    }
  };

% dotted continuation lines under third layer (left tree)
\foreach \nod in {a8,a9,a10,a11,a12,a13,a14,a15} {
  \draw[dotted] (\nod.south) -- ++(0,-6mm);
}

% -------------------------
% Second tree (right, shifted)
% -------------------------
\begin{scope}[xshift=6cm] % increase/decrease this to change spacing
\node (r2) {1}
  child { node {2}
    child { node {4}
      child { node (b8)  {8} }
      child { node (b10) [fill=red!30,draw=red] {10} } % highlighted swapped node
    }
    child { node {5}
      child { node (b9)  [fill=red!30,draw=red] {9} }  % highlighted swapped node
      child { node (b11) {11} }
    }
  }
  child { node {3}
    child { node {6}
      child { node (b12) {12} }
      child { node (b13) {13} }
    }
    child { node {7}
      child { node (b14) {14} }
      child { node (b15) {15} }
    }
  };

% dotted continuation lines under third layer (right tree)
\foreach \nod in {b8,b9,b10,b11,b12,b13,b14,b15} {
  \draw[dotted] (\nod.south) -- ++(0,-6mm);
}
\end{scope}

\end{tikzpicture}
\caption{A greedy splitting tree (left) and an eventually greedy splitting tree (right). In the second tree, the order of sticks 9 and 10 is exchanged, but the standard breadth-first indexing resumes thereafter. }\label{fig:1}
\end{figure}

\begin{theorem}
    The BET sequence is the unique eventually greedy stick-breaking sequence that satisfies the bounds
    in~(\ref{ineq:all3}).
\end{theorem}
\begin{proof}
Since the sequence is eventually greedy, there exists $N$
such that stick $n$ has children $2n$ and $2n+1$ whenever
$n\geq N$. Its descendants after $k$ generations are therefore
the sticks indexed by $2^kn,\ldots,2^k(n+1)-1$. Hence
\begin{equation}\label{grandchild}
\ell_n=\sum_{j=0}^{2^k-1}\ell_{2^kn+j}.
\end{equation}
By the previous lemma,
\[
\frac{1}{\log 2}
\sum_{j=0}^{2^k-1}\frac{1}{2^kn+j}
>
\ell_n
>
\frac{1}{\log 2}
\sum_{j=1}^{2^k}\frac{1}{2^kn+j}.
\]
As $k\to\infty$, both sums converge to
\[
\frac{1}{\log 2}\int_n^{n+1}\frac{dx}{x}
=\log_2\left(\frac{n+1}{n}\right).
\]
Thus, for every $n\geq N$,
\[
\ell_n=\log_2\left(\frac{n+1}{n}\right).
\]
In particular, after $n$ breaks the stick lengths are
\[
\log_2\left(\frac{n+1}{n}\right),\;
\ldots,\;
\log_2\left(\frac{2n}{2n-1}\right).
\]

We now reconstruct the earlier stages by gluing sticks.
Since $\log(1+x)<x$ for $x>0$, the second shortest stick
satisfies
\[
(n-1)\log_2\left(\frac{2n-1}{2n-2}\right)
<\frac{1}{2\log 2}.
\]
Thus, leaving either of the two shortest sticks unglued
would violate the lower bound on $(n-1)m_{n-1}(a)$.
Both must therefore be glued. Their combined length is
\[
\log_2\left(\frac{2n-1}{2n-2}\right)
+\log_2\left(\frac{2n}{2n-1}\right)
=\log_2\left(\frac{n}{n-1}\right).
\]
The resulting lengths are precisely
\[
\log_2\left(\frac{n}{n-1}\right),\;
\log_2\left(\frac{n+1}{n}\right),\;
\ldots,\;
\log_2\left(\frac{2n-2}{2n-3}\right),
\]
the BET lengths at the preceding stage.
Repeating this argument determines every earlier stage
uniquely.
\end{proof}

We do not know if the restriction of eventually greedy stick-breaking can be dropped. 
Ramshaw~\cite{ramshaw} proved that the BET sequence is not the only sequence satisfying 
the bounds once the strict inequalities in (\ref{ineq:all3}) are weakened to non-strict ones.

\section{Beyond stick-breaking}

A natural generalization of stick-breaking is to consider
breaking the stick in more than two pieces. 
Starting with a single piece, we repeatedly
choose one piece and divide it into $k$ smaller pieces.
We call this process \emph{$k$-shattering}.
Stick-breaking corresponds to $k=2$.

For a geometric example, consider the flat torus
obtained by identifying opposite sides of $[0,1)^2$.
Cutting along the two coordinate circles through the
origin opens the torus into a square. We then choose
a point inside the square and cut horizontally and
vertically through it, dividing the square into four
rectangles. Repeating this procedure, each time
subdividing one rectangle into four, produces
$1,4,7,10,\ldots$ tiles. This is an example of
$4$-shattering.

How small can we keep the ratio of the largest to
the smallest piece in a $k$-shattering process?
Related higher-dimensional splitting procedures
have been studied by Drmota and
Infusino~\cite{drmotaInfusino2012}.

\begin{lemma}[Generalized de Bruijn--Erd\H{o}s bounds]
\label{lem:generalized}
Let $(X,\mathcal F,\mu)$ be a probability measure space.
Starting with $X$, construct a sequence of measurable
partitions by repeatedly replacing one piece by $k$
pieces of positive measure, where $k\geq 2$.

Let $M_N$ and $m_N$ denote the largest and smallest
measures of the pieces when the partition has $N$
pieces. Thus $N\equiv 1\pmod{k-1}$.

For admissible $n$, define
\[
\Lambda_n=
\sup_{\substack{n/k<i\leq n\\
                 i\equiv 1\pmod{k-1}}}iM_i,
\qquad
\lambda_n=
\inf_{\substack{n/k<i\leq n\\
                 i\equiv 1\pmod{k-1}}}im_i.
\]
Then, for every $n\equiv 1\pmod{k-1}$,
\[
\Lambda_{kn-(k-1)}
\geq
\left(\sum_{j=0}^{n-1}
\frac{1}{n+(k-1)j}\right)^{-1}
\]
and
\[
\lambda_{kn}
\leq
\frac{1}{k}
\left(\sum_{j=0}^{n-1}
\frac{1}{kn-(k-1)j}\right)^{-1}.
\]
In particular,
\[
\limsup_{N\to\infty}NM_N
\geq \frac{k-1}{\log k},
\qquad
\liminf_{N\to\infty}Nm_N
\leq \frac{k-1}{k\log k}.
\]
\end{lemma}

\begin{proof}
Let $\ell_1\leq\ell_2\leq\cdots\leq\ell_n$
be the measures of the $n$ pieces, sorted in
ascending order. They add up to one.

Each split removes at most one of these original
pieces. After $j$ splits, at least $n-j$ original
pieces remain. Consequently,
\[
M_{n+(k-1)j}\geq\ell_{n-j},
\qquad j=0,\ldots,n-1.
\]
Summing gives
\[
\sum_{j=0}^{n-1}M_{n+(k-1)j}\geq 1.
\]
By definition,
\[
\Lambda_{kn-(k-1)}
\sum_{j=0}^{n-1}\frac{1}{n+(k-1)j}
\geq 1.
\]
This proves the first inequality.

For the second inequality, we work backwards,
gluing $k$ pieces together at each step.

Let $\ell_1\leq\ell_2\leq\cdots\leq\ell_{kn}$
be the measures of the pieces in a partition
with $kn$ pieces. After $j$ gluings, at most
$kj$ of these original pieces have disappeared.
Hence at least one of the $kj+1$ smallest
pieces remains, and therefore
\[
m_{kn-(k-1)j}\leq\ell_{kj+1},
\qquad j=0,\ldots,n-1.
\]
Since the measures are sorted,
\[
\sum_{j=0}^{n-1}\ell_{kj+1}\leq\frac1k.
\]
It follows that
\[
\sum_{j=0}^{n-1}m_{kn-(k-1)j}\leq\frac1k.
\]
By definition,
\[
\lambda_{kn}
\sum_{j=0}^{n-1}\frac{1}{kn-(k-1)j}
\leq\frac1k.
\]
This proves the second inequality.

Finally, both sums of reciprocals converge to
\[
\frac{\log k}{k-1}
\]
as $n\to\infty$. The asymptotic inequalities
follow.
\end{proof}

The lower bound on the ratio also extends to
$k$-shattering. In particular, for glass shattering
into four pieces at each step, the ratio of the largest to the
smallest tile area must exceed every constant
$c<4$ infinitely often.

\begin{lemma}\label{lem:ratio-generalized}
In the setting of the preceding lemma, let
$\rho_n=M_n/m_n$. For every $c<k$ and every $N$,
there exists an admissible $n\geq N$ such that
$\rho_n>c$. In particular,
$\limsup_{n\to\infty}\rho_n\geq k$.
\end{lemma}

\begin{proof}
When a piece is split into $k$ parts, at least one
of them has measure at most $1/k$ of the original.
Consequently,
\[
k m_{n+k-1}\leq M_n.
\]
Suppose that $\rho_n\leq c$ for all admissible
$n\geq N$. We may assume that $N$ is admissible.
Writing $n_j=N+j(k-1)$, we obtain
\[
k m_{n_{j+1}}\leq M_{n_j}
=\rho_{n_j}m_{n_j}.
\]
Iterating gives
\[
k^j m_{n_j}
\leq m_N\prod_{i=0}^{j-1}\rho_{n_i}.
\]
Since $M_{n_j}\geq 1/n_j$ and
$m_{n_j}=M_{n_j}/\rho_{n_j}$, it follows that
\[
\begin{aligned}
\frac{k^j}{N+j(k-1)}
&\leq k^j M_{n_j}\\
&\leq m_N\prod_{i=0}^{j}\rho_{n_i}\\
&\leq c^{j+1}.
\end{aligned}
\]
As $j\to\infty$, this is impossible when $c<k$.
\end{proof}

Recall that the BET sequence is given by
\[
b_n=\log_2(2n-1)\pmod 1.
\]
Its underlying construction repeatedly replaces a
logarithmic interval of length
\[
\log_2\frac{n+1}{n}
\]
by two intervals of lengths
\[
\log_2\frac{2n+1}{2n}
\quad\text{and}\quad
\log_2\frac{2n+2}{2n+1}.
\]
The two lengths add up to the original length.
We extend this construction to $k$-shattering by
replacing one interval by $k$ logarithmic intervals.

\begin{lemma}[Generalized BET construction]
\label{lem:generalized-BET}
For every integer $k\geq2$, there exists a
$k$-shattering process such that at every stage
\begin{equation}\label{ineq:kall3}
\boxed{
\rho_n(a)<k,\qquad
nM_n(a)<\frac{k-1}{\log k},\qquad
nm_n(a)>\frac{k-1}{k\log k}.}
\end{equation}
In particular, all three bounds in the preceding
lemmas are optimal.
\end{lemma}

\begin{proof}
Put $d=k-1$ and
\[
L_a=\log_k\frac{a+d}{a}.
\]
We begin by dividing $[0,1]$ into $k$ intervals
of lengths
\[
L_k,L_{k+d},\ldots,L_{k+(k-1)d}.
\]
Their sum is one, since
\[
\sum_{j=0}^{k-1}L_{k+jd}
=\log_k\frac{k+kd}{k}=1.
\]

At each subsequent step, we split the largest
interval according to
\[
L_a=\sum_{j=0}^{k-1}L_{ka+jd}.
\]
This identity follows by telescoping.

We claim that whenever there are $n$ intervals,
their lengths are
\[
L_n,L_{n+d},\ldots,L_{kn-d}.
\]
This holds for $n=k$. Since $L_a$ decreases
strictly with $a$, the largest interval has
length $L_n$. Splitting it removes $L_n$
and introduces
\[
L_{kn},L_{kn+d},\ldots,L_{kn+(k-1)d}.
\]
The resulting lengths are precisely
\[
L_{n+d},L_{n+2d},\ldots,L_{k(n+d)-d}.
\]
The claim follows by induction.

Consequently,
\[
M_n=\log_k\frac{n+d}{n},
\qquad
m_n=\log_k\frac{kn}{kn-d}.
\]
Using
\[
\log(1+x)<x,\qquad
-\log(1-x)>x
\quad (0<x<1),
\]
we obtain
\[
nM_n<\frac{d}{\log k}
<knm_n.
\]
Thus all three inequalities in
\eqref{ineq:kall3} hold.
\end{proof}

\section{Conclusion}

Some of our results first appeared in the first author's PhD thesis~\cite{brethouwer2024}, on mathematical models for conflict and crisis management, and in the analysis of the social distancing experiment~\cite{buijsrogge2026self}.

The authors used ChatGPT (OpenAI) to help improve the wording
and presentation of the manuscript. All suggestions were reviewed by the authors,
who take full responsibility for the final text and results.

\bibliographystyle{siam}
\bibliography{deBE}
\end{document}